\documentclass[12pt]{article}
\usepackage{amsmath}
\usepackage{amsfonts}
\usepackage{amsthm}
\usepackage{graphicx}
\graphicspath{{Figures/}} 
\usepackage{overpic}
\usepackage{amssymb} 
\usepackage{pictex}
\usepackage{rotating}  
\usepackage{xcolor}
\usepackage{cite} %pour liste de ref. remplac. par un tiret

\usepackage{enumitem} %si on veut enlever l'indent dans env. enumerate
\usepackage{accents} %pour avoir un cercle sur une lettre
\usepackage{tikz}
\usetikzlibrary{matrix}
\usetikzlibrary{arrows}
\usetikzlibrary{arrows.meta,calc,decorations.markings,math,arrows.meta}
\tikzset{
    set arrow inside/.code={\pgfqkeys{/tikz/arrow inside}{#1}},
    set arrow inside={end/.initial=>, opt/.initial=},
    /pgf/decoration/Mark/.style={
        mark/.expanded=at position #1 with
        {
            \noexpand\arrow[\pgfkeysvalueof{/tikz/arrow inside/opt}]{\pgfkeysvalueof{/tikz/arrow inside/end}}
        }
    },
    arrow inside/.style 2 args={
        set arrow inside={#1},
        postaction={
            decorate,decoration={
                markings,Mark/.list={#2}
            }
        }
    },
}
\usepackage{etex} %pour pb de \dimen
\definecolor{refkey}{rgb}{0,0,1}
\definecolor{labelkey}{rgb}{0,0,1}
\usepackage{comment}
\numberwithin{equation}{section}

\usepackage{etoolbox}

\newtheorem{theorem}{Theorem}%[section]
\newtheorem{proposition}[theorem]{Proposition}
\newtheorem{lemma}[theorem]{Lemma}

\newtheorem{Definition}[theorem]{Definition}
\newenvironment{definition}{\begin{Definition}\rm}{\end{Definition}}
\newtheorem{Remark}[theorem]{Remark}
\newenvironment{remark}{\begin{Remark}\rm}{\end{Remark}}
\newtheorem{Example}[theorem]{Example}

\newcommand{\R}{{\mathbb R}}
\newcommand{\om}{\omega}

\newcommand{\p}{\partial}

\let\oldforall\forall %To avoid infinite loop
\renewcommand{\forall}{\oldforall \, }
\let\oldexist\exists
\renewcommand{\exists}{\oldexist \: }

\def \capa{{\rm cap}\,}

\def\supp{\mathop{\mathrm{supp}}\nolimits}

\renewcommand{\bar}{\overline}
\renewcommand{\tilde}{\widetilde}

\usepackage[bookmarksopen, naturalnames]{hyperref}
\begin{document}

\title{Coulomb equilibrium in the external field of an attractive-repellent pair of charges}
\author{R. Orive and F. Wielonsky}

\maketitle 
\begin{abstract}
The aim of this paper is to provide a complete analysis of
the Coulomb equilibrium problem in the euclidean space $\R^d$, $d\geq2$, associated to the kernel $1/|x|^{d-2}$, with a non-convex external field created by an attractive-repellent pair of charges placed in $\R^{d+1} \setminus \R^d$.
We consider the \textit{admissible} setting, where the equilibrium measure is compactly supported, as well as the limiting \textit{weakly admissible} setting, with a weaker external field at infinity, where the existence of the equilibrium measure still holds but possibly with an unbounded support. The main tools for our analysis are the notions of signed equilibrium and
balayage of measures.
We note that for certain configurations of charges and distances to the conductor, the support of the equilibrium measure is a shell (multidimensional annulus).
\end{abstract}
%%%%%%%%%%%%%%%%%%%%
\section{Introduction}
We study the problem of minimizing the Coulomb energy of a probability measure $\mu$ in $\R^{d}$, $d\geq2$, that is the Coulomb equilibrium, in the radial and non-convex external field created by a pair $\gamma=(\gamma_{1},-\gamma_{2})$ of positive/negative charges in $\R^{d+1}$, located respectively at $H_{1}=(0,h_{1})$ and $H_{2}=(0,h_{2})$, at some heights $h_{1}$ and $h_{2}$ above the origin of $\R^{d}$. 

More precisely, we look for the equilibrium measure $\mu_{Q}$ in $\R^{d}$ minimizing the weighted energy $I^{Q}(\mu)$ of a probability measure $\mu$, given by
\begin{equation*}\label{weightedenergy}
I^{Q}(\mu) =
\begin{cases}
{\displaystyle\iint\frac{d\mu(x)d\mu(t)}{|x-t|^{d-2}}+\int Q(x)d\mu(x),\; Q(x)=\frac{\gamma_{1}}{|x-H_{1}|^{d-2}}
-\frac{\gamma_{2}}{|x-H_{2}|^{d-2}},\; d\geq3,}
\\[10pt]
{\displaystyle\iint\log\frac{1}{|x-t|}d\mu(x)d\mu(t)+\int Q(x)d\mu(x),\;
Q(x)=\log\frac{|x-H_{2}|^{\gamma_{2}}}{|x-H_{1}|^{\gamma_{1}}},
\quad d=2,}
\end{cases}
\end{equation*}
where $|\cdot|$ denotes the euclidean distance in $\R^{d}$ or $\R^{d+1}$. Note that, when $d=2$, Coulomb interaction is given by the logarithmic kernel $\log(1/|x|)$.

We need to assume that the confinement exerted by the external field $Q$ is sufficiently strong, namely that the strength of the negative charge is sufficiently larger than the one of the positive charge,
which translates into the inequality $\gamma_{1}-\gamma_{2}\leq-1$. This condition ensures existence and uniqueness of a solution, 
%with energy larger than $-\infty$, 
see \cite[Corollary 2.9]{DOSW1}. The limiting case $\gamma_{1}-\gamma_{2}=-1$ corresponds to the so-called weakly admissible setting, which was also considered in previous studies, and in different contexts, see e.g.\ \cite{S,HK,BLW,OSW,DOSW1}.

Our method of proof is based on the notion of signed equilibrium measure, see Section \ref{Prelim} for a precise definition, that was introduced originally in \cite{KD} for the logarithmic kernel in the complex plane, and which has been used recently in more general contexts, see e.g.\ \cite{OSW, DOSW1, DOSW2}. 
This method could be applied with a general distribution of masses along the vertical axis, instead of two single charges, but then with less explicit results.

Though the above problem can be seen as a problem taking place in $\R^{d+1}$, we consider Coulomb interaction with respect to the $d$-dimensional space $\R^{d}$.
We note that the concrete electrostatic problem in physics, posed in $\R^{d+1}$, actually corresponds to the Riesz potential $1/|x|^{d-1}$ in $\R^{d}$.

Probably, the most interesting situation in our analysis is that in which the support of the equilibrium measure is a shell (see Theorem 1 part (ii), below). As it is well known (see e.g. \cite{BHS}), the (unweighted) equilibrium measure of a shell is supported on its outer boundary for the Coulomb setting. In the present paper, we find situations where the support of the weighted equilibrium measure is a full annulus.
On the other hand, for the strictly subharmonic Riesz setting, when $\max (0,d-2)<s<d$, the unweighted equilibrium measure of a shell occupies the whole of it, but an explicit 
expression of its density is unknown %, except for $d=1$ 
(see \cite{Copson} or \cite{Clements} for some classical approaches to the electrostatic problem on a planar ring domain). More recently, it was conjectured in \cite{DOSW2} that the support of the equilibrium measure of a ball in a certain external field is a shell; but there it was only proved for the one-dimensional case.
It is also worth to mention that in \cite[Theorem I.1]{B} it is proved that for a radial external field $Q(x) = Q(|x|)$, which is convex in the radial direction, the support has to be an annulus (including the ball as a particular case); while in \cite[Theorem II.9]{B} a formula for the density of the equilibrium measure in a radial external field is given in terms of a complicated Fredholm integral equation of the second kind.
Finally, we note that
%, in the present paper, the external field is radial but non-convex, 
the consideration of the Coulomb kernel allows us to derive explicit formulas for all of the quantites under study.

Our main results are given in the following two theorems. First we consider admissible external fields $Q$ such that $-\gamma_{2}+\gamma_{1}<-1$.
\begin{theorem}\label{main} i) Assume 
%either $h_{2}< h_{1}$ or $h_{1}<h_{2}$ and 
$\gamma_{2}/\gamma_{1}\geq(h_{2}/h_{1})^{d}$. Then the equilibrium measure $\mu_{Q}$ is supported on the ball $B_{R_{s}}=\{x\in\R^{d},\,|x|\leq R_{s}\}$, whose radius $R_{s}$ satisfies the equation
\begin{equation}\label{sol-rad}
\frac{\gamma_{2}}{(1+(h_{2}/R_{s})^{2})^{d/2}}-\frac{\gamma_{1}}{(1+(h_{1}/R_{s})^{2})^{d/2}}=1
\quad\iff \quad \begin{cases}
R_{s}^{d-1}Q'(R_{s})=d-2,\quad d\geq3,
\\[5pt]
R_{s}Q'(R_{s})=1,\quad d=2.
\end{cases}
\end{equation}
It has a continuous density with respect to the Lebesgue measure in $\R^{d}$, given by
\begin{equation}\label{density1}
\frac{d\mu_{Q}}{dx}(x)=\frac{d}{|S^{d-1}|}\left(\frac{\gamma_2h_2^2}{(|x|^{2}+h_2^2)^{d/2+1}}-\frac{\gamma_1h_1^2}{(|x|^{2}+h_1^2)^{d/2+1}}\right),\qquad|x|\leq R_{s},
\end{equation}
where $|S^{d-1}|$ denotes the surface area of the $(d-1)$-dimensional unit sphere.\\
ii) Assume %$h_{1}< h_{2}$ and 
$\gamma_{2}/\gamma_{1}<(h_{2}/h_{1})^{d}$. The support of $\mu_{Q}$ is a shell 
$$\{x\in\R^{d},\,R_{0}\leq|x|\leq R_{s}\},$$ 
with outer radius $R_{s}$ given by (\ref{sol-rad}), density given by (\ref{density1}), and inner radius $0<R_{0}<R_{s}$ satisfying the equation
\begin{equation}\label{eq-inner}
\left(\frac{R_{0}^{2}+h_{1}^{2}}{R_{0}^{2}+h_{2}^{2}}\right)^{d}=\left(\frac{\gamma_{1}}{\gamma_{2}}\right)^{2}\quad\iff\quad Q'(R_{0})=0.
\end{equation}
When $\gamma_{2}/\gamma_{1}=(h_{2}/h_{1})^{d}$, which corresponds to the transition between the two previous cases, the density (\ref{density1}) vanishes at the origin. 
\end{theorem}
Note that when $h_{2}<h_{1}$, that is the negative charge is closer to $\R^{d}$ than the positive charge, i) takes place. For the statements in ii) to take place it is necessary that $h_{1}<h_{2}$.

For weakly external fields $Q$ such that $\gamma_{2}-\gamma_{1}=1$ the results are as follows.
\begin{theorem}\label{main2}
i) When $(\gamma_{2}/\gamma_{1})^{1/2}<h_{1}/h_{2}$ (which entails that $h_{2}<h_{1}$), the measure $\mu_{Q}$ is supported on a ball, with a radius and a density still satisfying (\ref{sol-rad}) and (\ref{density1}). 
\\
ii) When $(\gamma_1/\gamma_2)^{1/d} \leq h_1/h_2 \leq (\gamma_2/\gamma_1)^{1/2}$, 
$\mu_{Q}$ is supported on the whole space $\R^{d}$ with density given by (\ref{density1}).\\
iii) When $h_{1}/h_{2}<(\gamma_{1}/\gamma_{2})^{1/d}$  (which entails that $h_{1}<h_{2}$), the support of $\mu_{Q}$ is 
the complement of a ball of radius $R_{0}$ satisfying (\ref{eq-inner}). The density is given by (\ref{density1}).
\end{theorem}
These theorems complete the preliminary results that were obtained for more general Riesz kernels in \cite[Theorems 4.1 and 4.5]{DOSW1}.

The particular case $h_{2}=h_{1}:=h$ corresponds to a single negative charge of mass $-\gamma=-\gamma_{2}+\gamma_{1}\leq-1$ placed at the point $(0,h)\in\R^{d+1}$. It was studied for general Riesz kernels $1/|x|^{s}$, $d-2\leq s<d$, in \cite[Theorem 4.1]{DOSW1}, where the formulas for the Coulomb case $s=d-2$ were just derived 
%(without justification) 
as limits, as $s\downarrow d-2$, from those of the Riesz case. Since the formulas in \cite[Theorem 4.1]{DOSW1} simplify in the specific case $s=d-2$, we restate them here.
\begin{theorem}\label{main3} Assume a single negative charge of mass $-\gamma\leq-1$ is placed at the point $(0,h)\in\R^{d+1}$ above the origin of $\R^{d}$. If $\gamma>1$ then the support of $\mu_{Q}$ is a ball of radius $R_{s}$ satisfying
$$
\frac{\gamma}{(1+(h/R_{s})^{2})^{d/2}}=1\quad\iff\quad R_{s}=\frac{h}{\sqrt{\gamma^{2/d}-1}},
$$
and the density of $\mu_{Q}$ is given by
$$
\frac{d\mu_{Q}}{dx}(x)=\frac{d\gamma h^2}{|S^{d-1}|(|x|^{2}+h^2)^{d/2+1}}.
$$
If $\gamma=1$ then the support of $\mu_{Q}$ is the whole space $\R^{d}$ with a density still given by the expression above.
\end{theorem}
\begin{remark}
We may observe that in both theorems the support of the equilibrium measure $\mu_{Q}$ is always a connected set. Moreover, contrary to the behavior of the unweighted equilibrium measure, $\mu_{Q}$ never puts mass on the boundary of its support. Finally, its density with respect to the Lebesgue measure in $\R^{d}$ does not vanish on that boundary. The latter also signifies a distinct contrast with the strictly subharmonic setting when $\max (0,d-2) < s <d$, see \cite{KuML} for $d=1$ and $s=0$, and \cite{DOSW1} for $d\geq 2$ and $d-2<s<d$.
\end{remark}
%%%%%%%%%%%%%%%%%%%%
\begin{figure}\label{Fig-main2}
\begin{center}
\begin{tikzpicture}[scale=2.1]
\draw (0,-.1) -- (0,2) [arrow inside={end=stealth,opt={scale=3}}{1}];
\draw (-.25,0) -- (2,0) [arrow inside={end=stealth,opt={scale=3}}{1}];
%\draw (0,0) -- (1.9,1);
\draw (0,0) -- (1.5,1.9);
\draw (1.2,1.85) node {$L_{2}$};
\draw (-0.25,1.75) node {$h_{2}$};
\draw (1.75,-0.25) node {$h_{1}$};
\draw (0,0) node {$\bullet$};
\draw (0,-.25) node {$(0,0)$};
\draw[fill=gray](1.3,.7) circle (.15);
\draw[fill=gray](.5,1.5) circle (.15); %Anneau
\draw [fill=white](.5,1.5) circle (.07);
\end{tikzpicture}
 \hspace{1.8cm}
\begin{tikzpicture}[scale=2.1]
\draw (0,-.1) -- (0,2) [arrow inside={end=stealth,opt={scale=3}}{1}];
\draw (-.25,0) -- (2,0) [arrow inside={end=stealth,opt={scale=3}}{1}];
\draw (0,0) -- (1.9,1);
\draw (1.85,.75) node {$L_{1}$};
\draw (0,0) -- (1.5,1.9);
\draw (1.2,1.85) node {$L_{2}$};
\draw (-0.25,1.75) node {$h_{2}$};
\draw (1.75,-0.25) node {$h_{1}$};
\draw (0,0) node {$\bullet$};
\draw (0,-.25) node {$(0,0)$};
\draw[fill=gray](1.3,.3) circle (.15);
\draw (1,.8) node {$\R^{d}$};
\draw(.5,1.2) circle (.1); %Exterieur de la boule
 \foreach \angle in { 0,15,...,360 }{\draw [rotate around={\angle:(.5,1.2)}] (.6,1.2)  -- +(.1,0);}
\end{tikzpicture}
\end{center}
\caption{The different types of supports for the equilibrium measure according to the heights $h_{1}$ and $h_{2}$, the admissible case on the left, the weakly admissible case on the right. The line $L_{1}$ has equation $h_{2}=(\gamma_{1}/\gamma_{2})^{1/2}h_{1}$ and the line $L_{2}$ has equation $h_{2}=(\gamma_{2}/\gamma_{1})^{1/d}h_{1}$.}
\end{figure}
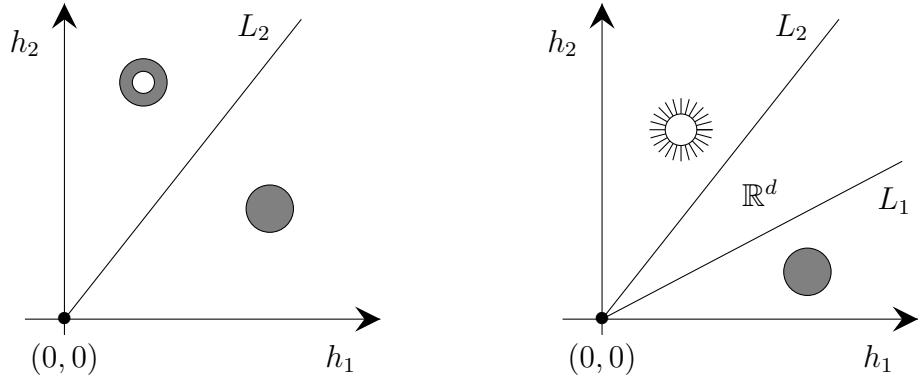
%Pour avoir des nombres encercles
\newcommand*\circled[1]{\tikz[baseline=(char.base)]{
    \node[shape=circle,draw,inner sep=1pt] (char) {#1};}}
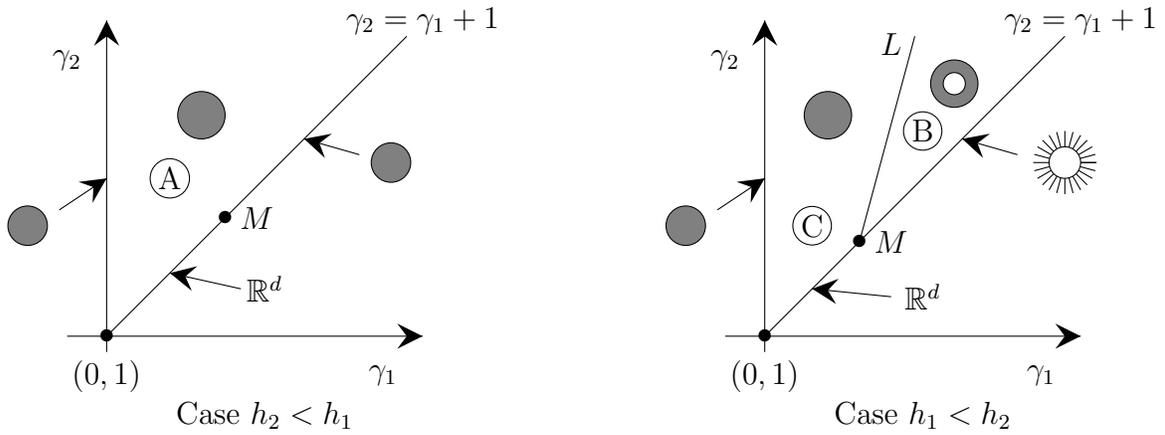
\begin{figure}
\begin{center}
\begin{tikzpicture}[scale=2.1]
\draw (0,-.1) -- (0,2) [arrow inside={end=stealth,opt={scale=3}}{1}];
\draw (-.25,0) -- (2,0) [arrow inside={end=stealth,opt={scale=3}}{1}];
\draw (0,0) -- (1.9,1.9);
\draw (2,2) node {$\gamma_{2}=\gamma_{1}+1$};
\draw (-0.25,1.75) node {$\gamma_{2}$};
\draw (1.75,-0.25) node {$\gamma_{1}$};
\draw (0,0) node {$\bullet$};
\draw (0,-.25) node {$(0,1)$};
\draw (.75,.75) node {$\bullet$}; \draw (.95,.75) node {$M$};
%\draw (.75,.75) -- (.95,1.9); %segment qui part du point
%\draw (.8,1.85) node {$L$};
\draw (1.6,1.15) -- (1.25,1.25) [arrow inside={end=stealth,opt={scale=3}}{1}];
\draw[fill=gray](1.8,1.1) circle (.125);
\draw (0.85,.3) -- (0.4,0.4) [arrow inside={end=stealth,opt={scale=3}}{1}];
\draw (1,.3) node {$\R^{d}$};
\draw (-0.3,.8) -- (0.,1) [arrow inside={end=stealth,opt={scale=3}}{1}];
\draw[fill=gray](-.5,.7) circle (.125);
\draw[fill=gray](.6,1.4) circle (.15);
%\draw[fill=gray](1.3,1.7) circle (.15);
\draw (.4,1) node {\circled{A}};
%\draw (.3,.65) node {\circled{B}};
\draw (1,-.5) node{Case $h_{2}<h_{1}$};
\end{tikzpicture}
\hspace{1.8cm}
\begin{tikzpicture}[scale=2.1]
\draw (0,-.1) -- (0,2) [arrow inside={end=stealth,opt={scale=3}}{1}];
\draw (-.25,0) -- (2,0) [arrow inside={end=stealth,opt={scale=3}}{1}];
\draw (0,0) -- (1.9,1.9);
\draw (2,2) node {$\gamma_{2}=\gamma_{1}+1$};
\draw (-0.25,1.75) node {$\gamma_{2}$};
\draw (1.75,-0.25) node {$\gamma_{1}$};
\draw (0,0) node {$\bullet$};
\draw (0,-.25) node {$(0,1)$};
\draw (.6,.6) node {$\bullet$}; \draw (.8,.6) node {$M$};
\draw (.6,.6) -- (.95,1.9); %segment qui part du point
\draw (.8,1.85) node {$L$};
\draw (1.6,1.15) -- (1.25,1.25) [arrow inside={end=stealth,opt={scale=3}}{1}];
%\draw (2,1.1) node {$\R^{d}\setminus B_{R}$}; %circle (.125);
\draw(1.9,1.1) circle (.1); %Exterieur de la boule
\foreach \angle in { 0,15,...,360 }{\draw [rotate around={\angle:(1.9,1.1)}] (2,1.1)  -- +(.1,0);}
 
\draw (0.8,.25) -- (0.3,0.3) [arrow inside={end=stealth,opt={scale=3}}{1}];
\draw (1,.25) node {$\R^{d}$};
\draw (-0.3,.8) -- (0.,1) [arrow inside={end=stealth,opt={scale=3}}{1}];
\draw[fill=gray](-.5,.7) circle (.125); %left circle
\draw[fill=gray](.4,1.4) circle (.15); 
\draw[fill=gray](1.2,1.6) circle (.15); %Anneau
\draw [fill=white](1.2,1.6) circle (.07);
\draw (1.0,1.3) node {\circled{B}};
\draw (.3,.7) node {\circled{C}};
\draw (1,-.5) node{Case $h_{1}<h_{2}$};
\end{tikzpicture}
\end{center}
\caption{The different types of supports for the equilibrium measure according to the values of $\gamma_{1}$ and $\gamma_{2}$, above the line $\gamma_{2}=\gamma_{1}+1$. Below that line, the external field $Q$ is not admissible and the equilibrium problem has no solution. The circled letters $A,B,C$ refer to the different cases defined in Section \ref{Proofs}, see (\ref{Cases}).}
\label{Fig-main}
\end{figure}
Figure \ref{Fig-main2} represents the different types of equilibrium that may occur, depending on the  values of the heights $h_{1}$ and $h_{2}$.

Figure \ref{Fig-main} displays the different types of equilibrium according to the values of the charges $\gamma_{1}$ and $\gamma_{2}$.
%, the case $h_{2}<h_{1}$ on the left, the case $h_{1}<h_{2}$ on the right. 
External fields $Q$ which are admissible correspond, on both pictures, to the sector above the line $\gamma_{2}=\gamma_{1}+1$, with that line corresponding to weakly admissible $Q$'s. As depicted, the support of $\mu_{Q}$ can be a ball, the entire space $\R^{d}$, a ring, or the complement of a ball. The point $M$ has coordinates 
$$
\left(\frac{h_{2}^{2}}{h_{1}^{2}-h_{2}^{2}},\frac{h_{1}^{2}}{h_{1}^{2}-h_{2}^{2}}\right)\quad\text{on the left},\qquad 
\left(\frac{h_{1}^{d}}{h_{2}^{d}-h_{1}^{d}},\frac{h_{2}^{d}}{h_{2}^{d}-h_{1}^{d}}\right)\quad\text{on the right}.
$$
The segment $L$ that starts from $M$ on the right picture is a subset of the line of equation
$\gamma_{2}=(h_{2}/h_{1})^{d}\gamma_{1}$.
The part of the $\gamma_{2}$-axis, above the point $(0,1)$, corresponds to the case of a single negative charge $-\gamma_{2}\leq-1$. The support of $\mu_{Q}$ is then always a ball, except at the boundary point $(0,1)$ where it is the entire space $\R^{d}$.
\begin{figure}[htb]
\includegraphics[width=7cm]{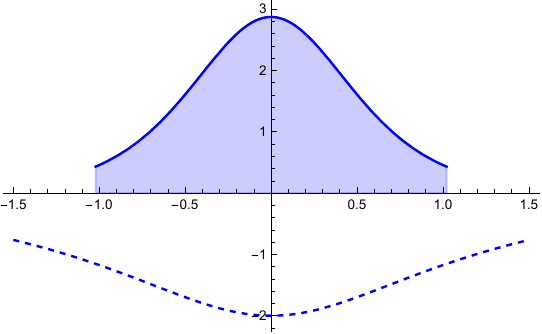}\hspace{1cm}
\includegraphics[width=7cm]{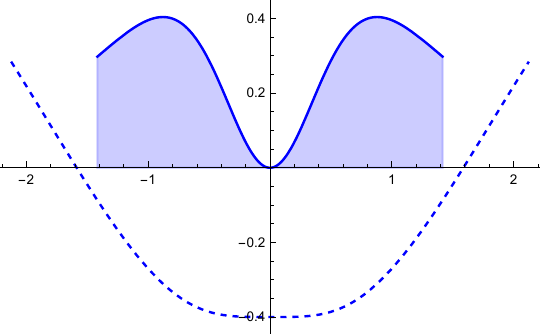}
\\[10pt]
\includegraphics[width=7cm]{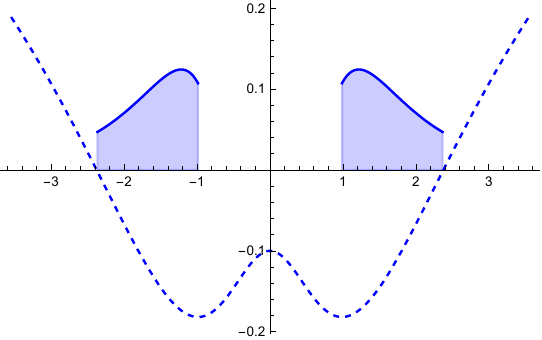}\hspace{1cm}
\includegraphics[width=7cm]{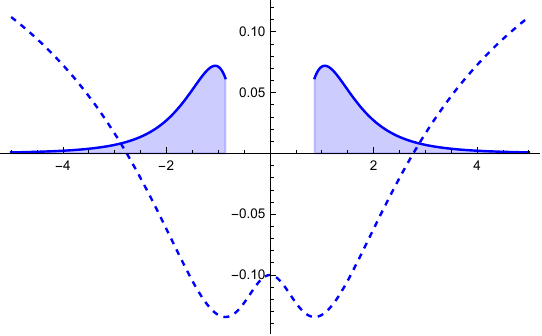}
\caption{Examples in $\R^{3}$ of radial densities of $\mu_{Q}$ when the support is a ball, a ball with a density vanishing at the origin (transition between a ball and a shell), a shell, and the complement of a ball. The external field $Q$ is represented by the dashed curve (up to an additive constant to fit the graph). Note that the densities do not vanish on the boundaries of the supports.}
\end{figure}
In Section \ref{Prelim} we give some preliminaries results useful for the sequel and in Section \ref{Proofs} we present the proofs of our main theorems.
%%%%%%%%%%%%
\section{Preliminaries}\label{Prelim}
In the sequel, we denote by \\[5pt]
-- $B_{R}^{d}$, or simply $B_{R}$, the closed $d$-dimensional ball in $\R^{d}$ of radius $R$,
%(or simply $B_{R}$ if there is no ambiguity on the dimension), 
\\[5pt]
-- $E_R$ the closure of the complement of $B_{R}$ in $\R^{d}$ (where $E$ stands for ``exterior''),
\\[5pt]
-- $S_{R}^{d}$, or simply $S_{R}$, the $d$-dimensional sphere of radius $R$, written $S^{d}$ or $S$ when $R=1$, \\[5pt]
-- $|S^{d}|$ the surface area of the $d$-dimensional unit sphere, \\[5pt]
-- $d\sigma_{R}^{d-1}$ (or $d\sigma_{R}$) the surface measure on $S_{R}^{d-1}$, \\[5pt]
-- $d\tilde\sigma_{R}^{d-1}$ (or $d\tilde\sigma_{R}$) the normalized surface measure on $S_{R}^{d-1}$, \\[5pt]
-- $\mu_{Q,R}$ the equilibrium measure on $B_{R}$ with external field $Q$, \\[5pt]
-- $\eta_{Q,R}$ the signed equilibrium measure on $B_{R}$ with external field $Q$, see Definition \ref{def-sign} below for that notion.
\\[5pt]\indent
%DEFINITION OF weighted equilibrium problem,\\
%SUFFICIENT COND FOR EXISTENCE,\\
The potential of a measure $\mu$ is defined as
$$
U^{\mu}(x)=\int\frac{d\mu(t)}{|x-t|^{d-2}},\quad d\geq3,\qquad
U^{\mu}(x)=\int\log\frac{1}{|x-t|}d\mu(t),\quad d=2,
$$
and its energy by
$$
I(\mu)=\int U^{\mu}(x)d\mu(x).
$$
The infimum of $I(\mu)$ among all probability measures $\mu$ supported on a closed set $\Sigma$ is the energy, denoted $W(\Sigma)$, of the set $\Sigma$. Its capacity is defined by 
$$\capa(\Sigma)=W(\Sigma)^{-1},\quad d\geq3,\qquad\capa(\Sigma)=e^{-W(\Sigma)},\quad d=2.
$$
When it exists, the equilibrium measure $\mu_{Q}$, minimizing the weighted energy $I^Q(\mu)$ for a general external field $Q$ on a closed set $\Sigma\subset\R^{d}$ of positive capacity, is unique and characterized by the Frostman inequalities, see e.g.\ \cite[Theorem 2.1]{DOSW1},
\begin{align}
U^{\mu_Q} (x) + Q(x)  & \geq F_{Q},\quad \textcolor{black}{\text{q.e. on}}\; \Sigma,\label{Frost1}\\
U^{\mu_Q} (x) + Q(x)  & \leq F_{Q},\quad x\in \supp(\mu_Q), \label{Frost2}
\end{align}
where q.e.\ is an abbreviation for quasi-everywhere, i.e.\ outside of a set of capacity zero, or equivalently a set of infinite energy. The finite constant
$$F_{Q}=I(\mu_{Q})+\int Qd\mu_{Q},$$
is referred to as the equilibrium constant or the (modified) Robin constant. 

We recall the notion of {\em balayage} of a measure, see \cite[Chapter IV, \S 1-2]{Lan}, and \cite[Theorems II.4.1 and II.4.4]{ST} when $d=2$. Given
%$d-2 \leq s <d$,
a closed set $F\subset \R^d$ 
%of positive capacity 
and a {positive} measure $\sigma$ \textcolor{black}{of finite total mass}, 
there exists a {positive} measure, which we denote $Bal(\sigma,F)$, the balayage of $\sigma$ onto $F$, satisfying $\supp(Bal(\sigma,F))\subset \p F$ and
%$$Bal(\sigma,F)$$called the Riesz $s$-balayage 
%$\bullet$ $S_{\widehat{\sigma}} \subseteq F$, \\
%$\bullet$ $\widehat{\sigma}$ is zero on the set of irregular points of \textcolor{black}{the complement of $F$}, and
\begin{align}\label{balayage} 
U^{Bal(\sigma,F)}(x) & = U^{\sigma}(x)+C\text{ q.e.\ on }F,
\\[5pt]
U^{Bal(\sigma,F)}(x) & \leq U^{\sigma}(x)+C\text{ q.e.\ on }\R^{d},
%,\qquad U^{\widehat{\sigma}}(x) \leq U^{\sigma}(x)\text{ on }\R^d.
\end{align}
where $C=0$ when $d\geq3$ or $d=2$ and $F$ has a bounded complement.
If $d=2$, the balayage preserves the total mass of the measure, but for $d\geq3$, in general, a mass loss occurs (for instance if $F$ is bounded, see e.g. \cite{DOSW1} and references therein).
%$\|\widehat{\sigma}_0\| = \|\sigma\|$, but this is not true for general $d-2 \leq s < d$.
%Actually, $\|\widehat{\sigma}\| = \|\sigma\|$ for all measures $\sigma$ if and only if $F$ is not thin at infinity, see \cite[Theorem 3.22]{FZ}, \cite[Corollary 5.3]{Z3}.
%In the logarithmic case, \eqref{balayage} is modified in the sense that
%\begin{equation*}%\label{logbalayage}
%U^{\widehat{\sigma}_{0}}(z) = U^{\sigma}(z) + C\quad \text{q.e. on }F,\qquad
%U^{\widehat{\sigma}_{0}}(z) \leq U^{\sigma}(z) + C\;\;\text{on }\C,
%\end{equation*}
%where $C=0$ if, for example, $\C \setminus F$ is a bounded set.

We now introduce one of the main tools for our analysis, namely the notion of signed equilibrium measure. It is a variation of the usual equilibrium measure, where the positivity of the measure is relaxed while the constraint on its potential is tightened. It was originally introduced in \cite{KD} for the logarithmic potential in the complex plane.
\begin{definition}\label{def-sign}
Let $\Sigma$ be a closed subset of $\R^d$. A {\it signed equilibrium measure} for $\Sigma$ in the external field $Q$ is a finite signed measure $\eta_{Q,\Sigma}$, with finite energy, supported on $\Sigma$ such that $\eta_{Q,\Sigma} (\Sigma) = 1$ and there exists a finite constant $C_{Q,\Sigma}$ such that
\begin{equation}\label{defsigned}
U^{\eta_{Q,\Sigma}}(x) + Q(x) = C_{Q,\Sigma}\quad\text{q.e. on }\Sigma.
\end{equation}
\end{definition}
It is known that, when it exists, the signed equilibrium measure is unique, see e.g.\ \cite[Lemma 23]{BDS}.
The next proposition describes some of its important properties. 
%that will be useful in the sequel.
The first two items are taken from
\cite[Lemma 3.15]{DOSW1}. The third and fourth items are new. They are variations of the third item in that reference.
\begin{proposition}\label{prop-eta}
Let $Q$ be a weight on $\Sigma$, a closed subset of $\R^d$ \textcolor{black}{of positive capacity}.
Assume an equilibrium measure $\mu_{Q,\Sigma}$ and a signed equilibrium measure $\eta_{Q,\Sigma}$ exist. Denote by $\eta_{Q,\Sigma}^+$ the positive part in the Jordan decomposition of $\eta_{Q,\Sigma}$.
%of the signed equilibrium measure $\eta_{Q,\Sigma}$.
 Then,
\\[5pt]
(i) One has
\begin{equation}\label{signeddomin}
\mu_{Q,\Sigma} \leq \eta_{Q,\Sigma}^+,\quad\text{and, in particular, }\quad
\supp(\mu_{Q,\Sigma}) \subseteq \supp(\eta_{Q,\Sigma}^+).
\end{equation}
(ii) \textcolor{black}{Let $\Sigma_{1}$ be a closed subset of $\Sigma$ that admits a signed equilibrium measure $\eta_{Q,\Sigma_1}$} for which $\supp(\mu_{Q,\Sigma})\subset\Sigma_{1}$. If $\eta_{Q,\Sigma_{1}}$ is a positive measure, then $\mu_{Q,\Sigma}=\eta_{Q,\Sigma_{1}}$.
\\[5pt]
(iii) Let $\Sigma=\R^{d}$ and let $Q$ be a radial external field such that $\eta_{Q,\Sigma}$ is negative near infinity. Assume that there exists an $R_{0}>0$ such that, for each $R>R_{0}$, the signed equilibrium measure $\eta_{Q,R}$ is negative on 
its boundary $S_{R}$, i.e.\ $\eta_{Q,R}(S_{R})<0$.
 Then $\supp(\mu_Q)\subset B_{R_{0}}$.
 \\[5pt]
(iv) Let $\Sigma=\R^{d}$ and let $Q$ be a radial external field such that $\eta_{Q,\Sigma}$ is negative near the origin. Assume that there exists an $R_{0}>0$ such that, for each $R<R_{0}$, the signed equilibrium measure $\eta_{Q,E_R}$ is negative on 
its boundary $S_{R}$, i.e.\ $\eta_{Q,E_R}(S_{R})<0$.
 Then $\supp(\mu_Q)\subset E_{R_{0}}$.
\end{proposition}
\begin{proof}
For a proof of i) and ii) see \cite[Lemma 3.15]{DOSW1}.\\
(iii): In view of (ii), the support of $\mu_{Q}$ is a compact set. Assume its outer boundary is a sphere $S_{R}$ with $R_{0}<R$. Then, $\mu_{Q}=\mu_{Q,{R}}\leq\eta_{Q,{R}}^{+}$ and thus $\mu_{Q}(S_{R})=0$ since $\eta_{Q,{R}}(S_{R})<0$. Since $\mu_{Q}(V)>0$ for any neighborhood $V$ of $S_{R}$, there exists $R_{0}<R'<R$ such that $S_{R'}\subset\supp\mu_{Q}$. Then,
$$
\mu_{Q,R'}=Bal(\mu_{Q},B_{R'})+(1-\|Bal(\mu_{Q},B_{R'})\|)\tilde\sigma_{R'}
$$
since $\supp Bal(\mu_{Q},B_{R'})=\supp\mu_{Q}\cap B_{R'}$ and thus $Bal(\mu_{Q},B_{R'})$ satisfies the Frostman inequalities (\ref{Frost1})-(\ref{Frost2}) on $B_{R'}$. Note that $\|Bal(\mu_{Q},B_{R'})\|<1$ and thus the second term in the right-hand side of the above equation is nonzero. Then, 
\begin{equation}\label{Bal1-R'}
0<\mu_{Q,R'}(S_{R'})\leq\eta_{Q,{R'}}^{+}(S_{R'})=0,
\end{equation}
where the last equality comes from the assumption since $R_{0}<R'$. We get a contradiction which shows that 
$\supp\mu_{Q}\subset B_{R_{0}}$. 
\\
(iv): Now, in view of (ii), the support of $\mu_{Q}$ does not contain the origin. Assume its inner boundary is a sphere $S_{R}$ with $R<R_{0}$. Then, $\mu_{Q}=\mu_{Q,E_R}\leq\eta_{Q,E_R}^{+}$ and thus $\mu_{Q}(S_{R})=0$ since $\eta_{Q,E_R}(S_{R})<0$. Since $\mu_{Q}(V)>0$ for any neighborhood $V$ of $S_{R}$, there exists $R<R'<R_{0}$ such that $S_{R'}\subset\supp\mu_{Q}$. Then,
\begin{equation}\label{Bal2-R'}
\mu_{Q,E_{R'}}=Bal(\mu_{Q},E_{R'})%+(1-\|Bal(\mu_{Q},B_{R'})\|)\tilde\sigma_{R'}
\end{equation}
since $\supp Bal(\mu_{Q},E_{R'})=\supp\mu_{Q}\cap E_{R'}$ and thus $Bal(\mu_{Q},E_{R'})$ satisfies the Frostman inequalities on $E_{R'}$. Note that (\ref{Bal2-R'}) does not involve a correction term as in (\ref{Bal1-R'}) since $\mu_{Q}$ is now swept out on $E_{R'}$ which contains a neighborhood of infinity, and thus its balayage has no mass loss.
%$\|Bal(\mu_{Q},B_{R'})\|<1$ and thus the second term in the right-hand side of the above equation is nonzero. 
Then, 
$$
0<\mu_{Q,E_{R'}}(S_{R'})\leq\eta_{Q,E_{R'}}^{+}(S_{R'})=0,
$$ 
where the first inequality comes from the fact that $\mu_{Q}$ has some mass outside of $E_{R'}$ which is swept out on $S_{R'}$, and the last equality comes from the assumption since $R'<R_{0}$. We get a contradiction which shows that 
$\supp\mu_{Q}\subset E_{R_{0}}$. 
%Item (iv) can be proved as (iii) in \cite[Lemma 3.15]{DOSW1}. For completeness, we do the proof.
%Assume $S_{\mu_Q}$ is not included in $E_{R_{0}}$. Consider the largest ball $B_{R}$, $R<R_{0}$, such that $S_{\mu_Q}\subset\tilde B_{R}$. Since $\mu_{Q}$ satisfies the Frostman inequalities %(\ref{Frostman3})--(\ref{Frostman4}) 
%on $\tilde B_{R}$, it holds that $\mu_{Q}=\mu_{Q,\tilde B_{R}}$, the weighted equilibrium measure of $\tilde B_{R}$, see e.g.\ \cite[Theorem 2.1]{DOSW1}. Now pick some $x$ in $S_{R}\cap S_{\mu_Q}$ (which equals $S_{R}$ by radial symmetry). In a small neighborhood $V$ of $x$, we have $\eta_{Q,\tilde B_{R}}(V)<0$ while $\mu_{Q,\tilde B_{R}}(V)>0$ which contradicts (\ref{signeddomin}).
\end{proof}
We will need an expression for the energy $W(S^{d})$ of the unit sphere $S^{d}\subset\R^{d+1}$ (i.e.\ the energy of its equilibrium measure) for the kernel $1/|x|^{d-2}$. According to \cite[Formula (4.6.5)]{BHS}, %we have
$$
W(S^{d})=\frac{2}{\sqrt{\pi}}\frac{\Gamma((d+1)/2)}{\Gamma(d/2+1)}=\frac{4}{d\sqrt{\pi}}\frac{\Gamma((d+1)/2)}{\Gamma(d/2)}
=\frac{4|S^{d-1}|}{d|S^{d}|}.
$$
Next, we recall, for the particular case of Coulomb interaction and $d>2$, an explicit formula for the balayage of a Dirac mass $\delta_{y}$, $y\in\R^{d+1}\setminus\R^{d}$, onto $\R^{d}$, that was given in \cite[Lemma 3.11]{DOSW1} for the more general case of Riesz potentials.
\begin{lemma}\label{lem:balayage}
Let $y= (y_1,\ldots,y_{d+1}) \in \R^{d+1} \setminus \R^d$, $d>2$, with $y_{d+1}\not= 0$. Define
the weak balayage $Bal(\delta_{y},\R^{d})$ of $\delta_{y}$ onto $\R^d$ as the measure of mass 1, given by
\begin{equation}\label{pointbalay}
dBal({\delta}_{y},\R^{d}) (x) = \frac{(2|y_{d+1}|)^{2}}{|S^{d}|W(S^{d}) |x - y|^{d+2}}dx=
\frac{d|y_{d+1}|^{2}}{|S^{d-1}| |x - y|^{d+2}}dx,
\end{equation}
where $dx$ denotes the Lebesgue measure in $\R^d$. Then $U^{Bal({\delta}_{y},\R^{d})}=U^{\delta_{y}}$ on $\R^{d}$.
\end{lemma}
Note that the ambient space in the above lemma is $\R^{d+1}$ for which the kernel $1/|x|^{d-2}$, we are interested in, is not the Coulomb kernel. This is why we just call $Bal({\delta}_{y},\R^{d})$ a weak balayage.

We complement this lemma with the case $d=2$ and the log kernel.
\begin{lemma}
Let $y= (y_1,y_{2},y_{3}) \in \R^{3} \setminus \R^2$ with $y_{3}\not= 0$.
The weak balayage $Bal(\delta_{y},\R^{2})$ of $\delta_{y}$ onto $\R^2$ is a measure of mass 1, given by
\begin{equation}\label{pointbalay-log}
dBal({\delta}_{y},\R^{2}) (x) = \frac{|y_{3}|^{2}}{|x - y|^{4}}\frac{dx}{\pi}.
\end{equation}
\end{lemma}
Note that (\ref{pointbalay-log}) is consistent with (\ref{pointbalay}).
\begin{proof}
We prove that the potentials of $\delta_{y}$ and $Bal({\delta}_{y},\R^{2})$ coincide on $\R^{2}$. Assume $y=(0,0,1)$ for simplicity. From the radial symmetry, it is sufficent to estimate the potential of the balayage at some $(R,0)\in\R^{2}$. We get
$$
-\int_{\R^{2}}\frac{\log|(R,0)-t|}{(1+|t|^{2})^{2}}\frac{dt}{\pi}=
2\int_{r=0}^{\infty}\frac{rU^{\tilde\sigma_{r}}(R)}{(1+r^{2})^{2}}{dr}.
$$
Recalling that
$$
U^{\tilde\sigma_{r}}(z)=\log1/|z|\quad  \text{if }|z|> r,\qquad
U^{\tilde\sigma_{r}}(z)=\log1/r\quad  \text{if }|z|\leq r,
$$
we obtain
\begin{align*}
2\int_{r=0}^{R} & \frac{r\log(1/R)}{(1+r^{2})^{2}}{dr}
+2\int_{r=R}^{\infty}\frac{r\log(1/r)}{(1+r^{2})^{2}}{dr}
=-\frac{R^{2}\log(R)}{(1+R^{2})}-\frac{\log(R)}{(1+R^{2})}
-\frac{1}{2}\log(1+1/R^{2})
\\[5pt]
& =\frac12\log\frac{1}{1+R^{2}}=U^{\delta_{y}}(R). \qedhere
\end{align*}
\end{proof}
We also need a formula for the balayage of a point mass in $\R^{d}$ onto a ball. 
%*** FIND A REF ? ***
\begin{lemma}
Let $d\geq2$ and $u\in\R^{d}$ with $|u|\neq R$. Then
\begin{align*}
Bal(\delta_{u},B_{R}^{c})(s)=\frac{(R^{2}-|u|^{2})d\sigma_{R}(s)}{R|u-s|^{d}|S^{d-1}|},\qquad|u|<R
,\\[5pt]
Bal(\delta_{u},B_{R})(s)=\frac{(|u|^{2}-R^{2})d\sigma_{R}(s)}{R|u-s|^{d}|S^{d-1}|},\qquad|u|>R.
\end{align*}
%with mass 
%$$
%\frac{|u|^{2}-R^{2}}{|S^{d-1}|R}\int_{S_{R}^{d-1}}\frac{d\sigma(s)}{|u-s|^{d}}
%=\frac{R^{d-2}}{|u|^{d-2}},
%$$
%where  and $|S^{d-1}|=\|d\sigma\|$ denotes the surface of the $(d-1)$-dimensional unit sphere.
\end{lemma}
\begin{proof}
The expression for $Bal(\delta_{u},B_{R}^{c})$ when $|u|<R$ follows from Poisson's formula in $B_{R}$. Indeed, for $|x|>R$,
$$
\frac{1}{|x-u|^{d-2}}=\int\frac{1}{|x-s|^{d-2}}\frac{R^{2}-|u|^{2}}{R|u-s|^{d}}\frac{d\sigma_{R}(s)}{|S^{d-1}|}.
$$
The expression when $|u|>R$ can be derived by making use of the Kelvin transform $x\mapsto(R^{2}/|x|^{2})x$.
\end{proof}
\begin{lemma}\label{int-d}
Let $d\geq2$. Then
$$
\frac{1}{|S^{d-1}|}\int_{S^{d-1}}\frac{d\sigma(s)}{|s-Re_{1}|^{d}}=
\begin{cases}
1/[R^{d-2}(R^{2}-1)]\quad & \text{ if }R>1,
\\[5pt]
1/(1-R^{2})\quad & \text{ if }R<1.
\end{cases}
$$
\end{lemma}
\begin{proof}
When $d=2$, the integral is easily computed by applying Cauchy's formula.\\
When $d\geq3$, denote by $I$ the left-hand side expression to be computed.
Assume first $R>1$. By harmonicity of $s\mapsto|s-Re_{1}|^{2-d}$, 
$$\frac{1}{|S^{d-1}|}\int_{S^{d-1}}\frac{d\sigma(s)}{|s-Re_{1}|^{d-2}}=R^{2-d},$$ 
and, by differentiating with respect to $R$, we get
\begin{multline*}
(2-d)R^{1-d}=\frac{2-d}{|S^{d-1}|}\int_{S^{d-1}}\frac{(R-<s,e_{1}>)d\sigma(s)}{|s-Re_{1}|^{d}}
=(2-d)(RI-J),\\[5pt]
J:=\frac{1}{|S^{d-1}|}\int_{S^{d-1}}\frac{<s,e_{1}>d\sigma(s)}{|s-Re_{1}|^{d}}.
\end{multline*}
Moreover,
\begin{multline*}
(1+R^{2})I-2RJ=
\frac{1}{|S^{d-1}|}\int_{S^{d-1}}\frac{(1-2R<s,e_{1}>+R^{2})d\sigma(s)}{|s-Re_{1}|^{d}}\\[5pt]
=
\frac{1}{|S^{d-1}|}\int_{S^{d-1}}\frac{d\sigma(s)}{|s-Re_{1}|^{d-2}}=R^{2-d}.
\end{multline*}
Combining the two equations, we get
\begin{equation*}
(1-R^{2})I=R^{2-d}-2R^{2-d}=-R^{2-d} %\qquad\iff\qquad I=\frac{1}{R^{d-2}(R^{2}-1)} 
\end{equation*}
and the result follows. If $R<1$, we may use the Kelvin transform $K:x\mapsto x/|x|^{2}$, which satisfies $|x-y|=|K(x)-K(y)||x||y|$ and whose jacobian has determinant 1 in absolute value on $S^{d-1}$.
\end{proof}
%%%%%%%%%%%%%%%%
\section{Proofs}\label{Proofs}
This section is devoted to the proofs of the main results. As said above, our main tool is the signed equilibrium measure. We start with some results about it.

\begin{proposition}
The signed equilibrium measure $\eta_{R}$ for the ball $B_{R}$, $R>0$, is made of two components, a continuous part supported on $B_{R}$ and a singular part supported on $S_{R}$, namely
\begin{equation}\label{form-signed}
\eta_{R}(x)=\frac{dg_{c}(|x|)}{|S^{d-1}|}dx
+g_{s}(R)d\tilde\sigma_{R}(s),
\end{equation}
with
\begin{equation}\label{def-g-h}
g_{c}(r) = \frac{\gamma_2h_2^2}{(r^{2}+h_2^2)^{d/2+1}}-\frac{\gamma_1h_1^2}{(r^{2}+h_1^2)^{d/2+1}},
\qquad g_{s}(R)=1-\frac{\gamma_2R^{d}}{(R^{2}+h_2^2)^{d/2}}+\frac{\gamma_1R^{d}}{(R^{2}+h_1^2)^{d/2}},
\end{equation}
where the indices $c$ and $s$ in the functions $g_{c}$ and $g_{s}$ stand for ``continuous'' and ``singular''. Note that the density $g_{c}(r)$ is independent from $R$.
\end{proposition}
\begin{proof}
%We first assume $d\geq3$.\\
1) The restriction of $Bal(\delta_{y},\R^{d})$ to $B_{R}$ is
$$
\frac{d|y_{d+1}|^{2}}{{|S^{d-1}|} |x - y|^{d+2}}dx.
%\quad\text{with mass}\quad
%\frac{(2|y_{d+1}|)^{2}}{{|S^{d}|}}\int_{B_{R}}\frac{dx}{ |x - y|^{d+2}}
%\left(=\frac{(2|y_{d+1}|)^{2}}{{|S^{d}|}}\right)
$$
2) The balayage of the restriction of $Bal(\delta_{y},\R^{d})$ to $B_{R}^{c}$ onto $B_{R}$ is supported onto $S_{R}^{d-1}$, and applying the \textit{superposition principle} \cite[Eq. (4.5.6)]{Lan}, it is equal to
\begin{equation}\label{bal-in}
\frac{d|y_{d+1}|^{2}}{R|S^{d-1}|} \left(\int_{B_{R}^{c}}\frac{1}{ |u - y|^{d+2}}
\frac{|u|^{2}-R^{2}}{|u-s|^{d}}du\right)\frac{d\sigma_{R}(s)}{|S^{d-1}|}.
\end{equation}
Because of radial symmetry the integral does not depend on $s$ and we may choose e.g. $s=Re_{1}$ with $e_{1}$ the first vector of the canonical basis of $\R^{d}$. Then, with Lemma \ref{int-d}, the integral becomes
\begin{align*}
\int_{\rho=R}^{\infty}\frac{(\rho^{2}-R^{2})\rho^{d-1}}{(\rho^{2}+y_{d+1}^{2})^{d/2+1}}
& \left(\int_{S^{d-1}}\frac{d\sigma(s)}{|\rho s-Re_{1}|^{d}}\right)d\rho
 =|S^{d-1}|\int_{\rho=R}^{\infty}\frac{(\rho^{2}-R^{2})\rho^{d-1}}
{(\rho^{2}+y_{d+1}^{2})^{d/2+1}\rho^{d-2}(\rho^{2}-R^{2})}
d\rho
\\[5pt]
& =|S^{d-1}|\int_{\rho=R}^{\infty}\frac{\rho}
{(\rho^{2}+y_{d+1}^{2})^{d/2+1}}
d\rho=\frac{|S^{d-1}|}{d(R^{2}+y_{d+1}^{2})^{d/2}}
\end{align*}
so that (\ref{bal-in}) equals
$$
\frac{|y_{d+1}|^{2}}{R|S^{d-1}|(R^{2}+y_{d+1}^{2})^{d/2}}d\sigma_{R}(s).
$$
\\
3) The mass of $Bal(\delta_{y},B_{R})$ is
$$
\frac{U^{\om_{B_{R}}}(y)}{W(B_{R})}=\frac{R^{d-2}}{R^{d-1}|S^{d-1}|}\int_{S_{R}^{d-1}}\frac{d\sigma_{R}(x)}{(|x|^{2}+y_{d+1}^{2})^{d/2-1}}
=\frac{R^{d-2}}{(R^{2}+y_{d+1}^{2})^{d/2-1}}.
$$
\\
Hence, taking into account the above-mentioned mass loss of the balayage for Riesz potentials, the signed equilibrium measure on $B_{R}$ corresponding to a Dirac mass $\gamma$ at $y=(0,y_{d+1})$ equals
\begin{multline*}
-\frac{d|y_{d+1}|^{2}\gamma}{{|S^{d-1}|} |x - y|^{d+2}}dx
-\frac{|y_{d+1}|^{2}\gamma}{R|S^{d-1}|(R^{2}+y_{d+1}^{2})^{d/2}}d\sigma_{R}(s)
  \\
+\left(1
%- \int_{B_{R}^{c}}\frac{(2|y_{d+1}|)^{2}}{{|S^{d}|} |u - y|^{d+2}}\frac{R^{d-2}}{|u|^{d-2}}du
  +\frac{\gamma R^{d-2}}{(R^{2}+y_{d+1}^{2})^{d/2-1}}\right)\frac{d\sigma_{R}(s)}{R^{d-1}|S^{d-1}|}.
\end{multline*}
In terms of $d\tilde\sigma_{R}$, we get
$$
-\frac{dy_{d+1}^{2}\gamma}{{|S^{d-1}|} (|x|^{2} +y_{d+1}^{2})^{d/2+1}}dx
+\left(1
  +\frac{\gamma R^{d}}{(R^{2}+y_{d+1}^{2})^{d/2}}\right)d\tilde\sigma_{R}(s).
$$
The formula (\ref{form-signed}) for the signed equilibrium measure corresponding to the pair of charges $(\gamma_{1},-\gamma_{2})$, respectively at heights $h_{1}$ and $h_{2}$ above $\R^{d}$, follows immediately.
%{\bf CASE $d=2$ : Same formula ?}
\end{proof}
To continue our analysis, we need a few properties of the functions $g_{c}$ and $g_{s}$ introduced in (\ref{def-g-h}).
\begin{lemma}\label{lem-gc-gs}
The functions $g_{c}$ and $g_{s}$ satisfy the following properties:\\
1) When $r=0$,
$$
g_{c}(0)=\frac{\gamma_{2}}{h_{2}^{d}}-\frac{\gamma_{1}}{h_{1}^{d}}\geq0\quad\iff\quad
\frac{\gamma_{2}}{\gamma_{1}}\geq\left(\frac{h_{2}}{h_{1}}\right)^{d},\qquad g_{s}(0)=1.
$$
2) When $r\to\infty$,
$$
g_{c}(r)\simeq\frac{\gamma_{2}h_{2}^{2}-\gamma_{1}h_{1}^{2}}{r^{d+2}}
%+\OO\left(\frac{1}{r^{d/2+2}}\right)
\geq0\quad\iff\quad\frac{\gamma_{2}}{\gamma_{1}}\geq\left(\frac{h_{1}}{h_{2}}\right)^{2},\qquad
g_{s}(r)\simeq1-\gamma_{2}+\gamma_{1}
%+\OO\left(\frac1r\right)
\leq0.
$$
3) The function $g_{c}$ has exactly one positive zero $r_{c}$ when 
$$h_{2}<h_{1}\quad\text{and}\quad(h_{2}/h_{1})^{d}<\gamma_{2}/\gamma_{1}<(h_{1}/h_{2})^{2},
$$
or
$$h_{1}<h_{2}\quad\text{and}\quad(h_{1}/h_{2})^{2}<\gamma_{2}/\gamma_{1}<(h_{2}/h_{1})^{d},$$ 
one zero at 0 when $\gamma_{2}/\gamma_{1}=(h_{2}/h_{1})^{d}$, 
and no zero otherwise.
The zero $r_{c}$ satisfies
\begin{equation}\label{sol-gc}
r_{c}^{2}=\frac{\alpha h_{1}^{2}-h_{2}^{2}}{1-\alpha}\qquad\text{with}\quad
\alpha:=\left(\frac{\gamma_2h_2^2}{\gamma_1 h_1^2}\right)^{2/(d+2)}.
\end{equation}
The function $g_{s}$ has precisely one positive zero $R_{s}$ when $1-\gamma_{2}+\gamma_{1}<0$ which is the unique positive real number satisfying
\begin{equation}\label{sol-gs}
\frac{\gamma_{2}}{(1+(h_{2}/R_{s})^{2})^{d/2}}-\frac{\gamma_{1}}{(1+(h_{1}/R_{s})^{2})^{d/2}}=1.
\end{equation}
In the limit case $1-\gamma_{2}+\gamma_{1}=0$, $g_{s}$ has still one zero $R_{s}$ when $1<\gamma_{2}/\gamma_{1}<(h_{1}/h_{2})^{2}$, and no zero when $\gamma_{2}/\gamma_{1}\geq(h_{1}/h_{2})^{2}$. From 2) we also have that $g_{s}(r)\to0$ as $r\to\infty$.
\end{lemma}
\begin{proof}
1) and 2) follows from the definitions of $g_{c}$ and $g_{s}$. \\
3) Solving $g_{c}(r)=0$ leads to
$$
\frac{r^{2}+h_2^2}{r^{2}+h_1^2}=\alpha
$$
which is equivalent to (\ref{sol-gc}). It is then easy to check the statements about the existence of the zero $r_{c}$.

For the function $g_{s}$, we have $g_{s}'(0)=0$ and it is easy to see that 
$$g_{s}'(r)=-dr^{d-1}g_{c}(r).$$ 
Hence $g_{s}'$ has at most 1 zero in $(0,\infty)$. Together with the sign change of $g_{s}$ between 0 and $\infty$, it implies that $g_{s}$ has exactly one zero in $(0,\infty)$ when $1-\gamma_{2}+\gamma_{1}<0$ and one or no zero when $1-\gamma_{2}+\gamma_{1}=0$, as stated in item 3). The equation (\ref{sol-gs}) follows directly from the expression for $g_{s}$.
\end{proof}
In view of Lemma \ref{lem-gc-gs}, it is quite natural to divide the study into the four following cases:
\begin{equation}\label{Cases}
\begin{cases}
\text{Case }A_{1}: h_{2}< h_{1} \quad\text{and}\quad1<\gamma_{2}/\gamma_{1}<(h_{1}/h_{2})^{2},\\
\text{Case }A_{2}: h_{2}< h_{1} \quad\text{and}\quad(h_{1}/h_{2})^{2}\leq\gamma_{2}/\gamma_{1},\\
\text{Case B}: h_{2}> h_{1} \quad\text{and}\quad1<\gamma_{2}/\gamma_{1}<(h_{2}/h_{1})^{d},\\
\text{Case C}: h_{2}> h_{1} \quad\text{and}\quad(h_{2}/h_{1})^{d}\leq\gamma_{2}/\gamma_{1}.
\end{cases}
\end{equation}
Note that this four cases correspond to the three domains shown in Figure \ref{Fig-main}, with $A_{1}$ and $A_{2}$ coresponding to the domain $A$.
\begin{figure}
\includegraphics[width=7cm]{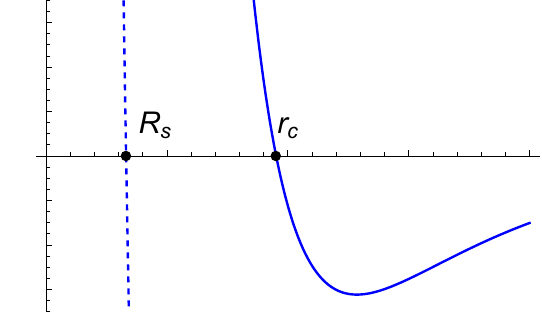}\hspace{1cm}
\includegraphics[width=7cm]{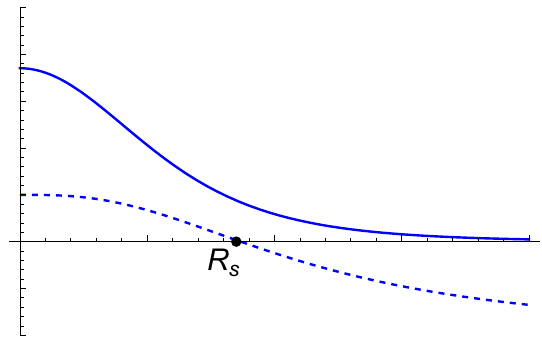}
\\[10pt]
\includegraphics[width=7cm]{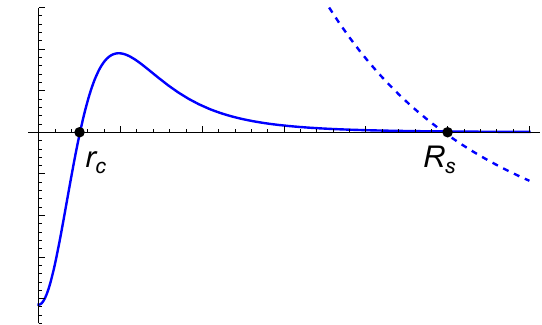}\hspace{1cm}
\includegraphics[width=7cm]{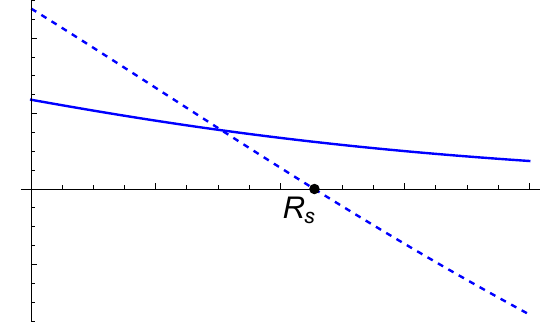}
\caption{Behavior of the densities $g_{c}$ (solid line) and $g_{s}$ (dashed line) in the four cases $A_{1},A_{2},B,C$ as defined in (\ref{Cases}), near the root $R_{s}$ of $g_{s}$, where it changes sign. The density $g_{c}$ has a root $r_{c}$ in cases $A_{1}$ and $B$, but no root in cases $A_{2}$ and $C$.}
\end{figure}
\begin{lemma}
Assume $g_{c}$ has a positive zero $r_{c}$ and $g_{s}$ has a positive zero $R_{s}$. Then
$$
R_{s}<r_{c}\quad\text{in case }A_{1},\qquad r_{c}<R_{s}\quad\text{in case B}.
$$
\end{lemma}
\begin{proof}
In case $A_{1}$, the function $g_{s}$ is decreasing in $(0,r_{c}]$ and increasing in $[r_{c},\infty)$. Hence we must have $R_{s}<r_{c}$. Note that the equality $R_{s}=r_{c}$ cannot happen since $R_{s}$ is a zero of order 1 of $g_{s}$. In case $B$, the function $g_{s}$ is increasing in $(0,r_{c}]$ and decreasing in $[r_{c},\infty)$. Hence we must have $r_{c}<R_{s}$. Again, note that the equality $R_{s}=r_{c}$ cannot happen since $R_{s}$ is a zero of order 1 of $g_{s}$.
\end{proof}
\begin{proof}[Proof of Theorems \ref{main} and \ref{main2}]
We first observe that in all of the four cases $A_{1},A_{2},B,C$, it holds that $\supp\mu_{Q}\subset B_{R_{s}}$. Indeed, in view of (\ref{form-signed}) and the fact that $g_{s}(R)<0$ for $R>R_{s}$, we have $\eta_{Q,R}(S_{R})<0$ for $R>R_{s}$, and (iii) of Proposition \ref{prop-eta} applies.
\\
-- {\bf Cases $A_{1},A_{2},C$}: The equilibrium problem $\mu_{Q}$ of our problem equals the signed equilibrium measure $\eta_{Q,R_{s}}$ on $B_{R_{s}}$. Indeed,
$\eta_{Q,{R_{s}}}$ is a positive measure and $\supp\mu_{Q}\subset B_{R_{s}}$ so that item (ii) of Proposition \ref{prop-eta} applies. Note that in case $A_{2}$, when $1-\gamma_{2}+\gamma_{1}=0$, one has $R_{s}=\infty$ and the support of the equilibrium measure is the whole space $\R^{d}$.\\
-- {\bf Case B}: Here the density $g_{c}$ is negative near the origin. Hence, we consider the balayage of the continuous part $dg_{c}(|x|)/|S^{d-1}|dx$ of the signed equilibrium measure $\eta_{Q,R_{s}}$ outside a ball of radius $R<R_{s}$. For a single unit charge at $y=(0,y_{d+1})$ we get a measure supported on the sphere $S_{R}$, 
%and has a density with respect to the normalized surface measure $d\tilde\sigma_{R}(s)$ 
given by
\begin{equation*}%\label{bal-in}
\frac{d|y_{d+1}|^{2}}{R|S^{d-1}|} \left(\int_{B_{R}}\frac{1}{ |u - y|^{d+2}}
\frac{R^{2}-|u|^{2}}{|u-s|^{d}}du\right)\frac{d\sigma_{R}(s)}{|S^{d-1}|},
\end{equation*}
which is an analog of the measure (\ref{bal-in}). By a computation completely similar to the one after (\ref{bal-in}) and with Lemma \ref{int-d}, one gets the measure
$$
\frac{R}{(R^{2}+y_{d+1}^{2})^{d/2}}\frac{d\sigma_{R}(s)}{|S^{d-1}|}=
\frac{R^{d}}{(R^{2}+y_{d+1}^{2})^{d/2}}d\tilde\sigma_{R}(s).
$$
For the pair of charges we are interested in we thus get for the balayage outside of $B_{R}$,
$$
\left(\frac{\gamma_{2}}{(R^{2}+h_{2}^{2})^{d/2}}-\frac{\gamma_{1}}{(R^{2}+h_{1}^{2})^{d/2}}\right)R^{d}d\tilde\sigma_{R}(s),
$$
which vanishes when $R=R_{0}$ such that
$$
\left(\frac{R_{0}^{2}+h_{1}^{2}}{R_{0}^{2}+h_{2}^{2}}\right)^{d}
=\left(\frac{\gamma_{1}}{\gamma_{2}}\right)^{2}.
$$
We already know that $\mu_{Q}$ is supported on $B_{R_{s}}$, and with item (iv) of Proposition \ref{prop-eta}, we derive that $\mu_{Q}$ is also supported on $E_{R_{0}}$, hence on the shell  bounded by the spheres $S_{R_{0}}$ and $S_{R_{s}}$. Since the signed equilibrium measure of that shell equals the balayage of $\eta_{Q,B_{R_{s}}}$ onto that shell, which is a positive measure, item (ii) of Proposition \ref{prop-eta} applies, which shows the results in case C. Note that in the limit case $1-\gamma_{2}+\gamma_{1}=0$ of weak admissibility, $R_{s}=\infty$ and the support of the equilibrium measure becomes the unbounded set $E_{R_{0}}$.
\end{proof}
\begin{proof}[Proof of Theorem \ref{main3}] 
Actually, when $\gamma>1$, this theorem can be seen as a particular case of Theorem \ref{main}. Indeed, it suffices to choose $\gamma_{2}=\gamma$, $\gamma_{1}=0$, $h_{2}=h$ and $h_{1}$ as any positive real number, for instance a number larger than $h_{2}=h$. Then the statements in Theorem \ref{main3} follows from those in i) of Theorem \ref{main}. The same observation holds, when $\gamma=1$, with Theorem \ref{main2}.
\end{proof}
\begin{remark}
As shown above, the signed equilibrium measure $\eta_{Q,R}$ used to solve the problem has a singular part on the boundary of the ball (or two, one on each boundary, in the case where the support is a shell). It plays an essential role in determining the support of the equilibrium measure $\mu_Q$. However, these singular measures, supported on spheres, which are typical for the unweighted equilibrium problems in the Coulomb case, do not appear in the expression of $\mu_Q$. We can thus think of them as ``ghost singular measures''.
\end{remark}
\begin{remark}
\cite[Proposition 2.13]{LG} for $d\geq3$ and \cite[Theorem IV.6.1]{ST} for $d=2$ with the log kernel, state the following. Assume that $Q(|x|)=Q(r)$ is a radial field satisfying one of the conditions:\\[5pt]
1) $r^{d-1}Q'(r)$ is increasing on $(0,\infty)$,\\[5pt]
2) $Q(r)$ is convex on $(0,\infty)$,\\[5pt]
and is such that $\mu_{Q}$ exists. 
Let $r_{0}$ be the smallest number for which $Q'(r)>0$ for all $r > r_{0}$, and $R_{0}$ the
smallest solution of $R^{d-1}Q'(R)=\max(1,d-2)$. 
Then the support of $\mu_{Q}$ is the shell $\{r_{0}\leq|x|\leq R_{0}\}$ (or ball if $r_{0}=0$).
Moreover, $\mu_{Q}$ is given by
$$
d\mu_{Q}(x)=\frac{1}{\max(1,d-2)}(r^{d-1}Q'(r))'drd\tilde\sigma(\bar x),\qquad x=r\bar x,~r=|x|.
$$
Condition 1) above is actually satisfied for our external field $Q$ for some values of the parameters $h_{1},h_{2},\gamma_{1},\gamma_{2}$. Indeed,
$$Q'(r) = \max (1,d-2)r\left(\gamma_2\,(r^2+h_2^2)^{-d/2} - \gamma_1\,(r^2+h_2^2)^{-d/2}\right),$$
and
$$(r^{d-1}\,Q'(r))' = \max (1,d-2)d r^{d-1}\left(\frac{\gamma_2 h_2^2}{(r^2+h_2^2)^{d/2+1}}-\frac{\gamma_1 h_1^2}{(r^2+h_1^2)^{d/2+1}}\right),$$
which is positive on $(0,\infty)$ when
\begin{equation*}\label{expression}
F(r):=\frac{r^2+h_1^2}{r^2+h_2^2}>\left(\frac{\gamma_1 h_1^2}{\gamma_2 h_2^2}\right)^{2/(d+2)}. 
\end{equation*}
Since $F(r)$ is a monotonic function of $r$, the above inequality is satisfied on $(0,\infty)$ if and only if
$$
\min(F(0),F(\infty))=\min(1,(h_{1}/h_{2})^{2})\geq\left(\frac{\gamma_1 h_1^2}{\gamma_2 h_2^2}\right)^{\frac{2}{d+2}}
~\iff~ \max\left(\left(\frac{h_{1}}{h_{2}}\right)^{2},\left(\frac{h_{2}}{h_{1}}\right)^{d}\right)
\leq\frac{\gamma_{2}}{\gamma_{1}},
$$ 
or, equivalently, when $\left({\gamma_1}/{\gamma_2}\right)^{1/d} \leq {h_1}/{h_2} \leq \left({\gamma_2}/{\gamma_1}\right)^{1/2}$,
which corresponds exactly to the cases $A_{2}$ and $C$ in (\ref{Cases}). 
\\
Concerning condition 2), the second derivative of $Q$ equals
$$
Q''(r)=\max(1,d-2)\left(\frac{\gamma_2 ((1-d)r^{2}+h_2^2)}{(r^2+h_2^2)^{d/2+1}}
-\frac{\gamma_1((1-d)r^{2} +h_1^2)}{(r^2+h_1^2)^{d/2+1}}\right)
$$
which cannot be positive on $(0,\infty)$ since, near infinity, $Q''(r)$ has the sign of $\gamma_{1}-\gamma_{2}$ which is negative.
In conclusion, the above-mentioned theorems can solve some of our cases but not all of them. In particular, the most interesting case where the equilibrium measure is supported on a shell cannot be recovered by these results.
\end{remark}
%\vspace{5cm}

\obeylines
\texttt{
R.\,Orive (rorive@ull.es)
Departmento de An\'{a}lisis Matem\'{a}tico, Universidad de La Laguna, 
38200 La Laguna (Tenerife), SPAIN.
\medskip
F.\,Wielonsky (franck.wielonsky@univ-amu.fr)
Aix Marseille Universit\'e, CNRS, I2M, Marseille, France
%Laboratoire I2M, UMR CNRS 7373, Universit\'e Aix-Marseille, 
%Campus Saint-Charles, 3 place Victor Hugo, Case 39
F-13453 Marseille Cedex 20, FRANCE.
}
\end{document}